\documentclass{article}
\usepackage{graphicx,latexsym,amsmath,amsthm,amssymb,color,verbatim,mathrsfs,bbm,amscd}
\usepackage[latin1] {inputenc}
\usepackage[english]{babel}
\usepackage{setspace}
\usepackage[enableskew]{youngtab}
\usepackage{leftidx}
\usepackage[all]{xy}
\usepackage{tikz,url}
\usetikzlibrary{matrix,decorations.pathreplacing}
\usepackage{floatrow}
\newcommand{\Sym}{\mathsf{Sym}}

\SelectTips{eu}{10}
\newcommand{\inj}{\ar@{^{(}->}}
\newcommand{\surj}{\ar@{->>}}

\newcommand{\la}{\lambda}
\newcommand{\IC}{\ensuremath{\mathbb{C}}}

\newcommand{\IN}{\ensuremath{\mathbb{N}}}

\newcommand{\aS}{\ensuremath{\mathfrak{S}}}
\newcommand{\aSS}{\ensuremath{\mathsf{S}}}

\DeclareMathOperator{\diag}{diag}

\newcommand{\Weyl}[1]{\{{#1}\}}
\newcommand{\Ch}{\textsf{Ch}}
\newcommand{\Specht}[1]{[{#1}]}
\newcommand{\tensor}{\smash{\textstyle\bigotimes}}
\newcommand{\GL}{\mathsf{GL}}

\swapnumbers
\numberwithin{equation}{section}
\newtheorem{theorem}[equation]{Theorem}
\newtheorem{corollary}[equation]{Corollary}

\newtheorem{conjecture}[equation]{Conjecture}
\newtheorem{claim}[equation]{Claim}

\theoremstyle{definition}
\newtheorem{remark}[equation]{Remark}

\title{On McKay's propagation theorem for the Foulkes conjecture}
\author{Christian Ikenmeyer${}^1$}
\date{\today}

\begin{document}
\sloppy
\maketitle
\begin{abstract}
We translate the main theorem in Tom McKay's paper ``On plethysm conjectures of Stanley and Foulkes'' (J.\ Alg.\ 319, 2008, pp.\ 2050--2071)
to the language of weight spaces and projections onto invariant spaces of tensors,
which makes its proof short and elegant.
\end{abstract}

{\mbox{~}}

{\noindent\footnotesize Keywords: Representation theory of the symmetric group; Plethysm; Foulkes Conjecture; Foulkes-Howe Conjecture

}

{\mbox{~}}

\noindent{\footnotesize MSC2010: 20C30

}

\footnotetext[1]{Texas A\&M University, ciken$@$math.tamu.edu}

\section{Introduction}
The Foulkes conjecture~\cite{Fou:50} states an inequality of certain representation theoretic multiplicities.
It comes in several different equivalent formulations, the most straightforward one being the following.
\begin{conjecture}[Foulkes conjecture]\label{conj:foulkes}
Let $a, b \in \IN_{>0}$ with $a \leq b$.
Let $U$ be a finite dimensional vector space of dimension at least $b$.
Then for every partition $\lambda$
the multiplicity of the irreducible $\GL(U)$ representation $\{\lambda\}$ in the plethysm
$\Sym^a(\Sym^b U)$ is at most as large as 
the multiplicity of $\{\lambda\}$ in $\Sym^b(\Sym^a U)$.
\end{conjecture}
The inequality $a \leq b$ is important: We know that $\Sym^a(\Sym^b U)$
contains irreducible $\GL(U)$ representations with up to $a$ parts,
but $\Sym^b(\Sym^a U)$ contains irreducible $\GL(U)$ representations with up to $b$ parts.

Using Schur-Weyl duality \cite{Gay:76} (see also \cite{Ike:12b}) one can interpret Conjecture~\ref{conj:foulkes} in terms of
representations of the symmetric group and we will use that interpretation later in this paper.

Conjecture~\ref{conj:foulkes} is true for $a \leq 5$:
for $a \leq 2$ see the explicit formulas in \cite{Thr:42},
for the case $a \leq 3$ see \cite{DS:00}, and see Corollary~\ref{cor:computation} for $a \leq 4$ and Corollary~\ref{cor:bettercomputation} for $a \leq 5$.
Brion \cite{Bri:93} showed that Conjecture~\ref{conj:foulkes} is true in the cases where $b$ is large enough with respect to $a$.
Conjecture~\ref{conj:foulkes} is true if we only consider partitions $\la$ with at most 2 rows~\cite{Her:1854}, a phenomenon called Hermite reciprocity.
Manivel showed that Conjecture~\ref{conj:foulkes} is also true for all partitions with very long first rows \cite[Thm.~4.3.1]{mani:98}.
In this case we do not only have an inequality of the multiplicities, but equality.

Conjecture~\ref{conj:foulkes} is equivalent to saying that there exists a $\GL(U)$ equivariant inclusion map
\begin{equation}\label{eq:foulkesconj}
\Sym^a(\Sym^b U) \hookrightarrow \Sym^b(\Sym^a U).
\end{equation}
A natural candidate for the map in \eqref{eq:foulkesconj} was the following map $\Psi_{a \times b}$:
\[
\xymatrix{
\Sym^a(\Sym^b U) \inj[d]^\iota \ar[r]^{\Psi_{a \times b}} & \Sym^b (\Sym^a U) \\
\tensor^a(\tensor^b U) \ar[r]^{r} & \tensor^b (\tensor^a U) \surj[u]_{\varrho}
}
\]
where $\iota$ denotes the canonical embedding of symmetric tensors in the space of all tensors,
$\varrho$ is the canonical projection,
and $r$ is the canonical isomorphism given by reordering tensor factors.
This map has a natural analog $\psi_{a\times b}$ in the symmetric group interpretation that we will discuss in Section~\ref{sec:prelim}.
It was combinatorially defined in \cite{BL:89}.

Hadamard conjectured \cite{Had:97} that $\Psi_{a\times b}$ is injective for all $a \leq b$.
Howe \cite[p.~93]{How:87} wrote that this ``is reasonable to expect''.
However, M\"uller and Neunh\"offer \cite{MN:05} showed that $\Psi_{5\times 5}$ has a nontrivial kernel.
In \cite{CIM:15} the kernel of $\Psi_{5\times 5}$ is determined as a $\GL(U)$ representation: It is multiplicity free and consists of the following types of irreducible representations:
(14, 7, 2, 2), (13, 7, 2, 2, 1), (12, 7, 3, 2, 1), (12, 6, 3, 2, 2),
(12, 5, 4, 3, 1), (11, 5, 4, 4, 1), (10, 8, 4, 2, 1), and (9, 7, 6, 3).
Also it is shown in \cite{CIM:15} that $\Psi_{6\times 6}$ is not injective. 

McKay \cite{McK:08} contributed the following main theorem.
\begin{theorem}[McKay's main theorem]\label{thm:main}
If $\Psi_{a \times (b-1)}$ is injective, then $\Psi_{a \times c}$ is injective for all $c \geq b$.
\end{theorem}
The proof uses the analog map $\psi_{a \times b}$, defined in Section~\ref{sec:prelim},
and decomposes it into a composition of two maps (Section~\ref{sec:decomp}) whose injectivity is proved independently (Sections~\ref{subsec:leftfactor} and Sections~\ref{subsec:rightfactor}).

Theorem~\ref{thm:main} allows us to verify Conjecture~\ref{conj:foulkes} in infinitely many cases while only doing a finite calculation:
\cite{MN:05} calculate that $\Psi_{4 \times 4}$ is injective, so Theorem~\ref{thm:main} implies the following corollary.
\begin{corollary}\label{cor:computation}
Conjecture~\ref{conj:foulkes} is true in all cases where $a = 4$.
\end{corollary}
Although $\Psi_{5\times 5}$ is not injective, if $a=b$ then $\Sym^a(\Sym^b U) = \Sym^b(\Sym^a U)$.
The recent calculation \cite{CIM:15} reveals that $\Psi_{5 \times 6}$ is injective, therefore Theorem~\ref{thm:main} implies the following corollary.
\begin{corollary}\label{cor:bettercomputation}
Conjecture~\ref{conj:foulkes} is true in all cases where $a = 5$.
\end{corollary}

\begin{remark}
Let $\Ch_b:=\overline{\GL_b (X_1 X_2 \cdots X_b)} \subseteq \Sym^b \IC^b$ denote the $\GL_b$ orbit closure of the monomial $X_1 X_2 \cdots X_b$.
We call this orbit closure the $b$th Chow variety.
The kernel of $\Psi_{a \times b}$ is known to be the homogeneous degree $a$ part of the vanishing ideal of $\Ch_b$,
as was shown by Hadamard, see e.g.~\cite[Section~8.6]{Lan:11}.
\end{remark}

For more information on the history of the Foulkes conjecture and the kernel of $\Psi_{a \times b}$ we refer the interested reader to \cite[Section~7.1]{Lan:15}.

\subsection*{Acknowledgments}
I thank JM Landsberg for bringing McKay's paper to my attention, for valuable discussions, and for providing help with the historical background.

\section{Preliminaries}\label{sec:prelim}
Fix natural numbers $a, b \in \IN_{>0}$ and let $V:=\IC^a \oplus \IC^b$.
The group $\GL_a \times \GL_b$ acts canonically on $\IC^a \oplus \IC^b$
and hence on the $ab$th tensor power $\tensor^{ab}V$.
Let $\aS_a$ be the symmetric group on $a$ letters and embed $\aS_a \subseteq \GL_a$ via permutation matrices.
Let $a \times b := (b,b,\ldots,b)$ denote the partition of $ab$ whose Young diagram is a rectangle with $a$ rows and $b$ columns (in anglophone notation).
Let $\emptyset$ denote the empty partition, i.e., $\emptyset := 0 \times 0$.
For complex numbers $s_1,\ldots,s_a$ let $\diag(s_1,\ldots,s_a)$ denote the diagonal matrix with $s_1,\ldots,s_a$ on the main diagonal.
Analogously for $\diag(t_1,\ldots,t_b)$.
For $\alpha \in \IN^a$ and $\beta\in \IN^b$ the set of tensors $(\tensor^{ab}V)_{\alpha,\beta} := $
\[
\{w \in \tensor^{ab}V \mid (\diag(s_1,\ldots,s_a),\diag(t_1,\ldots,t_b))w = s_1^{\alpha_1}\cdots s_a^{\alpha_a} t_1^{\beta_1}\cdots t_b^{\beta_b} w\}
\]
is called the $(\alpha,\beta)$ weight space of $\tensor^{ab}V$.
Here $\alpha$ and $\beta$ might be partitions, but could also be \emph{weak compositions},
i.e., we do not require the entries of $\alpha \in \IN_{\geq 0}^a$ and $\beta\in\IN_{\geq 0}^b$ to be ordered.
The weight space $(\tensor^{ab}V)_{a \times b,\emptyset}$ is closed under the action of $\aS_a$.
Let $(\tensor^{ab}V)_{a \times b,\emptyset}^{\aS_a}$ denote the $\aS_a$ invariant space.
Analogously define $(\tensor^{ab}V)_{\emptyset,b \times a}^{\aS_b}$.
On the space of tensors $\tensor^{ab} V$ we have the canonical action of $\aS_{ab}$ via permutation of the tensor factors.
To avoid confusion with $\aS_a$ or $\aS_b$, we use the symbol $\aSS_{ab}$ for the group $\aS_{ab}$ if it acts by permuting the tensor factors.
Since the actions of $\aSS_{ab}$ and $\GL(V)$ commute, $(\tensor^{ab}V)_{a \times b}^{\aS_a}$ and $(\tensor^{ab}V)_{b \times a}^{\aS_b}$ are $\aSS_{ab}$ representations.
Recall that the irreducible $\aSS_{ab}$ representations $\Specht\la$ are indexed by partitions $\la$ whose Young diagrams have $ab$ boxes, i.e., $|\la|=ab$.

Let $e_1,\ldots,e_a$ denote the standard basis of $\IC^a$
and let $f_1,\ldots,f_b$ denote the standard basis of $\IC^b$.
For $d \in \IN$, $W:=\tensor^d V$,
$1 \leq i \leq a$ and $1 \leq j \leq b$ we define $\varphi_{i,j}:W\to W$
to be the raising operator that projects each $(\alpha,\beta)$ weight space in $W$
to the $(\alpha-e_i,\beta+f_j)$ weight space.
For example $\varphi_{2,1}$ maps the $((3,2,1),(2,2))$ weight space to the $((3,1,1),(3,2))$ weight space.
\begin{claim}\label{cla:commute}
All the $\varphi_{i,j}$ commute.
\end{claim}
We postpone the proof to Section~\ref{subsec:proofs}.

Define the map $\varphi_{a \times b}: (\tensor^{ab}V)_{a \times b,\emptyset} \to (\tensor^{ab}V)_{\emptyset,b \times a}$ via
\begin{equation}\label{eq:def:phiaxb}
\varphi_{a \times b} := \varphi_{1,1} \circ \varphi_{1,2} \circ \cdots \varphi_{1,b} \circ \varphi_{2,1} \circ \cdots \circ \varphi_{2,b} \circ \cdots \cdots \circ \varphi_{a,b}.
\end{equation}
Note that according to Claim~\ref{cla:commute} the order of the factors in \eqref{eq:def:phiaxb} does not matter.
The restriction of $\varphi_{a \times b}$ to the linear subspace of $\aS_a$ invariants shall be denoted by
\[
\psi_{a \times b} : (\tensor^{ab}V)_{a\times b,\emptyset}^{\aS_a} \to (\tensor^{ab}V)_{\emptyset,b\times a}.
\]
It is easy to see that $\psi_{a \times b}$ actually maps to $(\tensor^{ab}V)_{\emptyset,b\times a}^{\aS_b}$,
but we will omit this detail in the upcoming proofs.
Since each $\varphi_{i,j}$ is $\aSS_{ab}$ equivariant, the map $\psi_{a\times b}$ is $\aSS_{ab}$ equivariant.

Using Schur-Weyl duality we see that the multiplicity of $\{\la\}$ in $\Sym^a(\Sym^b U)$ equals
the multiplicity of the irreducible $\aSS_{ab}$ representation $[\la]$ in $(\tensor^{ab}V)_{a\times b,\emptyset}^{\aS_a}$
and the multiplicity of $\{\la\}$ in $\Sym^b(\Sym^a U)$ equals
the multiplicity of $[\la]$ in $(\tensor^{ab}V)_{b\times a,\emptyset}^{\aS_b}$.
Moreover, the multiplicity of $\{\la\}$ in the kernel of $\Psi_{a \times b}:\Sym^a(\Sym^b U)\to\Sym^b(\Sym^a U)$
is precisely the multiplicity of $[\la]$ in the kernel of $\psi_{a\times b}$.
Therefore we can phrase McKay's Theorem~\ref{thm:main} as follows:
\begin{theorem}[McKay's main theorem]\label{thm:mainsn}
If $\psi_{a \times (b-1)}$ is injective, then $\psi_{a \times c}$ is injective for all $c \geq b$.
\end{theorem}

The rest of this paper is devoted to the proof of Theorem~\ref{thm:mainsn}.
Using induction it suffices to prove it for the case $b=c$.
The proof goes by decomposing $\psi_{a \times b}$ into a composition of two maps (Section~\ref{sec:decomp})
and then proving injectivity for the left factor (Section~\ref{subsec:leftfactor}) and the right factor (Section~\ref{subsec:rightfactor}) separately.

\section{Decomposition of the canonical map}\label{sec:decomp}
Let $a < b$ and set $B := b-1$ to simplify notation.
We interpret $\GL_B \subseteq \GL_b$ as $B \times B$ matrices embedded in the upper left corner of $b \times b$ matrices with an additional 1 at the lower right corner.
Let $V:=\IC^a \oplus \IC^b$.
\begin{claim}\label{cla:mapstoinv}
Given a tensor power $W:=\tensor^d V$.
The composition
\[
\varphi_{1,b} \circ \varphi_{2,b} \circ \cdots \circ \varphi_{a,b}: W \to W
\]
maps $\aS_a$-invariants to $\aS_a$-invariants.
\end{claim}
We postpone the proof to Section~\ref{subsec:proofs}.

Let $(0^B,a)$ denote the weak composition $(0,0,...,0,a) \in \IN^b$ that is zero everywhere but in the last entry.

Consider the map $\psi_{a \times b} : (\tensor^{ab}V)_{a\times b,\emptyset}^{\aS_a} \to (\tensor^{ab}V)_{\emptyset,b\times a}$.
According to Claim~\ref{cla:mapstoinv} the right factor of
\begin{equation}\label{eq:leftrightfactor}
\psi_{a \times b} = (\underbrace{\varphi_{1,1} \circ \cdots \circ \varphi_{1,B} \circ \varphi_{2,1} \circ \cdots \circ \varphi_{a,B}}_{\text{left factor}}) \circ (\underbrace{\varphi_{1,b} \circ \cdots \circ \varphi_{a,b}}_{\text{right factor}}).
\end{equation}
maps $\aS_a$ invariants to $\aS_a$ invariants, so that we can write
\[
(\tensor^{ab}V)_{a\times b,\emptyset}^{\aS_a} \stackrel{\text{right factor}}{\longrightarrow} (\tensor^{ab}V)_{a\times B,(0^B,a)}^{\aS_a} \stackrel{\text{left factor}}{\longrightarrow} (\tensor^{ab}V)_{\emptyset,b\times a}.
\]
The left factor is similar to $\psi_{a \times B}$, but with a larger domain of definition.
The proof idea for Theorem~\ref{thm:mainsn} is to prove injectivity of both factors independently, where the injectivity of the left factor will follow from the induction hypothesis that $\psi_{a \times B}$ is injective.

\section{The left factor}\label{subsec:leftfactor}
We want to be more precise about the relationship between the left factor of \eqref{eq:leftrightfactor} and $\psi_{a \times B}$.
Let $\IC^B \subseteq \IC^b$ be embedded as vectors that have a zero as their last component.
Let $V':= \langle e_1,\ldots,e_a,f_1,\ldots,f_B\rangle = \IC^a \oplus \IC^B \subseteq V$ be the complement of the 1-dimensional vector space spanned by the basis vector~$f_b$.
We decompose $\tensor^{ab}V$ as follows. $\tensor^{ab}V = $
\begin{equation}\label{eq:decomposeleft}
\bigoplus_{Q \subseteq [ab]} \langle \{ v_{1} \otimes v_{2} \otimes \cdots \otimes v_{ab} \mid v_{i} \in V' \text{ if } i \notin Q, \ v_{i}=f_b \text{ if } i \in Q \}\rangle,
\end{equation}
where $[ab]:=\{1,2,\ldots,ab\}$ and $\langle\ \rangle$ denotes the linear span.
We denote by $\tensor^{ab}_Q V$ the summand in \eqref{eq:decomposeleft} corresponding to $Q$.
The weight spaces split as follows:
\[
(\tensor^{ab}V)_{a\times B,(0^B,a)}^{\aS_a} = \bigoplus_{{Q \subseteq [ab]}\atop{|Q|=a}} (\tensor^{ab}_Q V)_{a\times B,(0^B,a)}^{\aS_a}
\]
and
\[
(\tensor^{ab}V)_{\emptyset,b \times a} = \bigoplus_{{Q \subseteq [ab]}\atop{|Q|=a}} (\tensor^{ab}_Q V)_{\emptyset,b \times a}.
\]
As a $\GL_a \times \GL_B$ representation, $\tensor^{ab}_Q V$ is canonically isomorphic to $\tensor^{ab-|Q|} V'$.
Using this isomorphism, for $|Q|=a$ we see that the following diagram commutes:
\[
(\tensor^{ab}_Q V)_{a\times B,(0^B,a)}^{\aS_a} \simeq (\tensor^{aB} V')_{a\times B,\emptyset}^{\aS_a}
\]
\[
\psi_{a \times B}^Q \downarrow \quad \quad \quad \quad \quad \quad \quad \downarrow \psi_{a \times B}
\]
\[
\ \ \ \ \ \ \ (\tensor^{ab}_Q V)_{\emptyset,b \times a} \simeq (\tensor^{aB} V')_{\emptyset,B \times a}
\]
where $\psi_{a \times B}^Q$ is the left factor of \eqref{eq:leftrightfactor} restricted to $(\tensor^{ab}_Q V)_{a\times B,(0^B,a)}^{\aS_a}$.
Hence we have
\[
\psi_{a \times b} = (\bigoplus_{{Q \subseteq [ab]} \atop {|Q|=a}} \psi_{a \times B}^Q) \circ (\varphi_{1,b}\circ \varphi_{2,b} \circ \cdots \circ \varphi_{a,b}).
\]
We see that since $\psi_{a \times B}$ is injective by induction hypothesis, each $\psi_{a \times B}^Q$ is injective,
and $\psi_{a \times b}$ is injective as a direct sum of injective maps whose ranges form a direct sum.

\section{The right factor}\label{subsec:rightfactor}
To show that the right factor of \eqref{eq:leftrightfactor} is injective it suffices to show injectivity for the $a$ factors $\varphi_{i,b}$, $1 \leq i \leq a$.
Let $\gamma_i$ denote the weak composition
\[
\gamma_i := (\underbrace{B,B,\ldots,B}_{i \text{ times}}, \underbrace{b,b,\ldots,b}_{a-i \text{ times}})
\]
whose Young diagram has $ab-i$ boxes (if we flip the rows, then it is a partition).
We have $\gamma_0=a \times b$ and $\gamma_a=a \times B$.
With this notation we can write
\[
(\tensor^{ab} V)_{a \times b,\emptyset} \hspace{10cm}
\]
\[
\rotatebox{90}{=} \hspace{10cm}
\]
\[
(\tensor^{ab} V)_{\gamma_0,(0^B,0)} \stackrel{\varphi_{1,b}}{\to} (\tensor^{ab} V)_{\gamma_1,(0^B,1)} \stackrel{\varphi_{2,b}}{\to} (\tensor^{ab} V)_{\gamma_2,(0^B,2)} \stackrel{\varphi_{3,b}}{\to} \cdots \stackrel{\varphi_{a,b}}{\to} (\tensor^{ab} V)_{\gamma_a,(0^B,a)}
\]
\[
\hspace{11cm} \rotatebox{90}{=}
\]
\[
\hspace{10.5cm} (\tensor^{ab} V)_{a \times B,(0^B,a)}
\]
Note that formally we are only required to prove the injectivity of this chain
$\varphi_{a,b}\circ\cdots\circ\varphi_{1,b}$
restricted to the $\aS_a$ invariant space $(\tensor^{ab} V)_{a \times b,\emptyset}^{\aS_a}$.
We ignore the action of $\aS_a$ and in the following claim we prove the injectivity of each factor of the chain, which finishes the proof of Theorem~\ref{thm:mainsn}.
\begin{claim}
Let $a < b$ and $1 \leq i \leq a$.
The map
\[
\varphi_{i,b} : (\tensor^{ab} V)_{\gamma_{i-1},(0^B,i-1)} \to (\tensor^{ab} V)_{\gamma_{i},(0^B,i)}
\]
is injective.
\end{claim}
\begin{proof}
Fix $i$ with $1 \leq i \leq a$.
Note that the inequality $i \leq B$ holds, because $i \leq a < b$. We will use it later.

We proceed similary to the proof for the left factor.
Let $V'=\langle e_1,\ldots,e_{i-1},e_{i+1},\ldots,e_a,f_1,f_2,\ldots,f_B\rangle$
be the complement of $\langle e_i,f_b\rangle$.
For a subset $Q \subseteq [ab]$ we denote by $\tensor^{ab}_Q V$ the vector space spanned by
\[
\{ v_{1} \otimes \cdots \otimes v_{ab} \mid v_{i} \in V' \text{ if } i \in Q, v_{i} \in \langle e_i,f_b\rangle \text{ if } i \notin Q \}.
\]
The weight spaces split as follows:
\[
(\tensor^{ab} V)_{\gamma_{i-1},(0^B,i-1)} = \bigoplus_{Q \subseteq [ab]\atop{|Q|=ab-(B+i)}} (\tensor^{ab}_Q V)_{\gamma_{i-1},(0^B,i-1)}
\]
and
\[
(\tensor^{ab} V)_{\gamma_{i},(0^B,i)} = \bigoplus_{Q \subseteq [ab]\atop{|Q|=ab-(B+i)}} (\tensor^{ab}_Q V)_{\gamma_{i},(0^B,i)}.
\]
We embed $\GL_2 \subseteq \GL(V)$ to be the $\GL_2$ that preserves the linear space $\langle e_i,f_b\rangle$.
As a $\GL_2$ representation $\tensor^{ab}_QV$ is canonically isomorphic to $\tensor^{ab-|Q|}\IC^2$.
Using this isomorphism, for $|Q|=ab-(B+i)$ we see that the following diagram commutes:
\[
(\tensor^{ab}_Q V)_{\gamma_{i-1},(0^B,i-1)} \simeq (\tensor^{B+i} \IC^2)_{b,i-1}
\]
\[
\varphi_{i,b}^Q \downarrow \quad \quad \quad \quad \quad \quad \quad \downarrow \zeta
\]
\[
(\tensor^{ab}_Q V)_{\gamma_{i},(0^B,i)} \simeq (\tensor^{B+i} \IC^2)_{B,i}
\]
where $\varphi_{i,b}^Q$ is the raising operator $\varphi_{i,b}$
restricted to the weight space $(\tensor^{ab}_Q V)_{\gamma_{i-1},(0^B,i-1)}$,
and
$\zeta : (\tensor^{B+i} \IC^2)_{b,i-1} \to (\tensor^{B+i} \IC^2)_{B,i}$ is the canonical $\GL_2$ raising operator.

It remains to show that $\zeta$ is injective,
because then all $\varphi_{i,b}^Q$ are injective and hence
$\varphi_{i,b} = \bigoplus_Q \varphi_{i,b}^Q$ is injective as a direct sum of injective maps whose ranges form a direct sum.
We will use the fact that $i \leq B$.

By Schur-Weyl duality,
\[
\tensor^{B+i}\IC^2 = \bigoplus_{\lambda} \{\la\} \otimes [\la]
\]
as a $\GL_2 \times \aSS_{B+i}$ representation, where $\lambda$ runs over all partitions with $B+i$ boxes and at most 2 parts,
and $\{\la\}$ denotes the irreducible $\GL_2$ representation of type~$\la$.
The Kostka number determines the dimension of the $(b,i-1)$ weight space in $\Weyl{\la}$:
It is 1-dimensional iff $i-1 \geq \la_2$ and zero otherwise.
Therefore as an $\aSS_{B+i}$ representation we have
\[
(\tensor^{B+i}\IC^2)_{(b,i-1)} = \bigoplus_{\lambda \atop {i-1 \geq \la_2}} [\la]. 
\]
By Schur's lemma, since $\zeta$ is $\aSS_{B+i}$ equivariant, it suffices to check for each partition $\lambda$ that
a single $(b,i-1)$ weight vector in the $\la$-isotypic $\GL_2$ component is not mapped to zero by $\zeta$.
We will choose an explicit weight vector as follows.
Let $\{e,f\}$ be the standard basis of $\IC^2$.
For $B > i-1 \geq \la_2$ we calculate
\[
(e \wedge f)^{\la_2} \otimes (e^{b-\la_2} \cdot f^{i-1-\la_2}) \quad \stackrel{\zeta}{\mapsto} \quad (e \wedge f)^{\la_2} \otimes (e^{B-\la_2} \cdot f^{i-\la_2}) \ \neq \ 0,
\]
where $\cdot$ denotes the symmetric product.
\end{proof}

\section{Appendix: Proofs of the preliminary claims}\label{subsec:proofs}
\begin{proof}[Proof of Claim~\protect{\ref{cla:commute}}]
The map $\varphi_{i,j}$ is defined on basis tensors as follows:\\
$\varphi_{i,j}(v_1 \otimes \cdots \otimes v_d) = $
\[
\tfrac 1 d\Big( \zeta_{i,j}(v_1) \otimes v_2 \otimes \cdots \otimes v_d + v_1 \otimes \zeta_{i,j}(v_2) \otimes \cdots  \otimes v_d + \cdots + v_1 \otimes \cdots  \otimes \zeta_{i,j}(v_d)\Big),
\]
where all $v_k \in \{e_1,\ldots,e_a,f_1,\ldots,f_b\}$,
and $\zeta_{i,j}:V \to V$ maps $e_i$ to $f_j$ and vanishes on all other basis vectors.
Since $\zeta_{i,j} \circ \zeta_{i',j'} = 0$,
for the composition of maps $\varphi_{i,j} \circ \varphi_{i',j'}$ we have
\[
(\varphi_{i,j}\circ \varphi_{i',j'})(v_1 \otimes \cdots \otimes v_d) = \tfrac 1 {d(d-1)} \sum_{{p,q=1}\atop{p\neq q}}^d
v_1 \otimes v_2 \otimes \zeta_{i,j}(v_p) \otimes \cdots \otimes \zeta_{i',j'}(v_q) \otimes \cdots \otimes v_d.
\]
This expression is symmetric in $p$ and $q$.
Therefore $\varphi_{i,j}\circ \varphi_{i',j'}=\varphi_{i',j'}\circ \varphi_{i,j}$.
\end{proof}
\begin{proof}[Proof of Claim~\protect{\ref{cla:mapstoinv}}]
The action of $\aS_a$ permutes the weight spaces. More precisely, for $\pi \in\aS_a$ and $w \in W$ we have
\[
\pi \varphi_{i,b}(w) = \varphi_{\pi(i),b}(\pi w).
\]
Therefore, if we take $w$ to be $\aS_a$-invariant, we see that
\[
\pi\Big((\varphi_{1,b} \circ \cdots \circ \varphi_{a,b})(w)\Big) = 
(\underbrace{\varphi_{\pi(1),b} \circ \cdots \circ \varphi_{\pi(a),b}}_{\stackrel{\ref{cla:commute}}{=}\varphi_{1,b} \circ \cdots \circ \varphi_{a,b}})(\underbrace{\pi w}_{=w})
\]
for every $\pi \in \aS_a$, and hence $(\varphi_{1,b} \circ \cdots \circ \varphi_{a,b})(w)$ is $\aS_a$-invariant.
\end{proof}


\end{document}